\documentclass[12pt]{amsart}
\usepackage[dvips]{graphicx}
\usepackage{amssymb,amsmath,amsthm}

\def\NZQ{\mathbb}               
\def\NN{{\NZQ N}}

\def\frk{\mathfrak}               

\def\Phi{{\frk N}}
\def\Pc{{\mathcal P}}

\def\I{{\mathcal I}}

\def\opn#1#2{\def#1{\operatorname{#2}}} 
\opn\chara{char} \opn\length{\ell} \opn\pd{pd} \opn\rk{rk}
\opn\projdim{proj\,dim} \opn\injdim{inj\,dim} \opn\rank{rank}
\opn\depth{depth} \opn\grade{grade} \opn\height{height}
\opn\size{size}
\opn\embdim{emb\,dim} \opn\codim{codim}

\opn\Tr{Tr} \opn\bigrank{big\,rank}
\opn\superheight{superheight}\opn\lcm{lcm}
\opn\trdeg{tr\,deg}
\opn\reg{reg} \opn\lreg{lreg} \opn\ini{in} \opn\lpd{lpd}
\opn\size{size}\opn{\mult}{mult}
\opn{\Cl}{Cl}
\opn\div{div} \opn\Div{Div} \opn\cl{cl} \opn\Cl{Cl}
\opn\Spec{Spec} \opn\Supp{Supp} \opn\supp{supp} \opn\Sing{Sing}
\opn\Ass{Ass} \opn\Min{Min} \opn\cl{cl}
\opn\Ann{Ann} \opn\Rad{Rad} \opn\Soc{Soc}
\opn\Syz{Syz} \opn\Im{Im} \opn\Ker{Ker} \opn\Coker{Coker}
\opn\Am{Am} \opn\Hom{Hom} \opn\Tor{Tor} \opn\Ext{Ext}
\opn\End{End} \opn\Aut{Aut} \opn\id{id} \opn\ini{in}

\opn\nat{nat}
\opn\pff{pf}
\opn\Pf{Pf} \opn\GL{GL} \opn\SL{SL} \opn\mod{mod} \opn\ord{ord}
\opn\Gin{Gin}
\opn\Hilb{Hilb}\opn\adeg{adeg}\opn\std{std}\opn\ip{infpt}
\opn\Pol{Pol}
\opn\sat{sat}
\opn\Var{Var}
\opn\Gen{Gen}
\opn\lex{lex}
\opn\div{div}

\opn\aff{aff} \opn\con{conv} \opn\relint{relint} \opn\st{st}
\opn\lk{lk} \opn\cn{cn} \opn\core{core} \opn\vol{vol}
\opn\link{link} \opn\star{star}
\opn\gr{gr}

\def\Ic{{\mathcal I}}

\def\Cc{{\mathcal C}}

\def\pot#1#2{#1[\kern-0.28ex[#2]\kern-0.28ex]}

\opn\dirlim{\underrightarrow{\lim}}
\opn\inivlim{\underleftarrow{\lim}}
\def\Implies{\ifmmode\Longrightarrow \else
        \unskip${}\Longrightarrow{}$\ignorespaces\fi}
\def\implies{\ifmmode\Rightarrow \else
        \unskip${}\Rightarrow{}$\ignorespaces\fi}
\def\iff{\ifmmode\Longleftrightarrow \else
        \unskip${}\Longleftrightarrow{}$\ignorespaces\fi}

\let\:=\colon

\def\NZQ{\mathbb}        
\def\NN{{\NZQ N}}

\def\frk{\mathfrak}        

\def\Phi{{\frk N}}

\def \P{{\mathcal P}}
\def\C{{\mathcal C}}
\newtheorem{Theorem}{Theorem}[section]
\newtheorem{Lemma}[Theorem]{Lemma}
\newtheorem{Corollary}[Theorem]{Corollary}

\theoremstyle{definition}

\newtheorem{Example}[Theorem]{Example}

\begin{document}
\title{Simple polyominoes are prime}
\author {Ayesha Asloob Qureshi, Takafumi Shibuta and Akihiro Shikama}

\thanks{}


\address{Ayesha Asloob Qureshi, Department of Pure and Applied Mathematics, Graduate School of Information Science and Technology,
Osaka University, Toyonaka, Osaka 560-0043, Japan}
\email{ayesqi@gmail.com}
\address{Takafumi Shibuta, Institute of Mathematics for Industry, Kyushu University, Fukuoka 819-0395, Japan}
\email{shibuta@imi.kyushu-u.ac.jp}
\address{Akihiro Shikama, Department of Pure and Applied Mathematics, Graduate School of Information Science and Technology,
Osaka University, Toyonaka, Osaka 560-0043, Japan}
\email{a-shikama@cr.math.sci.osaka-u.ac.jp}

\maketitle
\begin{abstract}
In this paper we show that polyomino ideal of a simple polyomino coincides with the toric ideal of a  weakly chordal bipartite  graph and hence it has a quadratic Gr\"obner basis with respect to a suitable monomial order.
\end{abstract}

\section*{Introduction}
Polyominoes are two dimensional objects which are originally rooted in recreational mathematics and combinatorics. They have been widely discussed in connection with tiling problems of the plane. Typically, a polyomino is plane figure  obtained by joining squares of equal sizes, which are known as cells. In connection with  commutative algebra, polyominoes were first discussed in \cite{Q} by assigning each polyomino the ideal of its inner 2-minors or the {\em polyomino ideal}. The study of ideal of $t$-minors of an $m \times n$ matrix is a classical subject in commutative algebra.  The class of polyomino ideal widely generalizes the class of ideals of 2-minors of $m \times n $ matrix as well as the ideal of inner 2-minors attached to a two-sided ladder.

Let $\P$ be a polyomino and $K$ be a field. We denote by $I_{\P}$, the polyomino ideal attached to $\P$, in a suitable polynomial ring over $K$. The residue class ring defined by $I_{\Pc}$ is denoted by $K[\P]$. It is natural to investigate the algebraic properties of $I_{\P}$  depending on shape of $\P$. In \cite{Q}, it was shown that for a convex polyomino, the residue ring $K[\P]$ is a normal Cohen-Macaualay domain. More generally, it was also shown that polyomino ideals attached to a row or column convex polyomino is also a prime ideal. Later in \cite{EHH}, a classification of the convex polyominoes whose polyomino ideals are linearly related is given. For some special classes of polyominoes, the regularirty of polyomino ideal is discussed in \cite{ERQ}.

In \cite{Q}, it was conjectured that polyomino ideal attached to a simple polyomino is prime ideal.  Roughly speaking, a simple polyomino is a polyomino without 'holes'. This conejcture was further studied in \cite{HQS}, where authors introduced {\em balanced} polyominoes and proved that polyomino ideals attached to balanced polyominoes are prime. They expected that all simple polyominoes are balanced, which would then prove simple polyominoes are prime. This question was further discussed in \cite{HM}, where authors proved that balanced and simple polyominoes are equivalent. Independent of the proofs given in \cite{HM}, in this paper we show that simple polyominoes are prime by identifying the attached residue class ring $K[\P]$ with the edge rings of weakly chordal graphs. Moreover, from \cite{OH1}, it is known that toric ideal of the  edge ring of weakly chordal bipartite graph has a quadratic Gr\"obner basis with respect to a suitable monomial order, which implies that $K[\P]$ is Koszul.

\section{Polyominoes and Polyomino ideals}
First we recall some definitions and notation from \cite{Q}.  Given $a=(i,j)$ and $b=(k,l)$ in $\NN^2$ we write  $a\leq b$ if $i\leq k$ and $j\leq l$. The set $[a,b]=\{c\in\NN^2\:\; a\leq c\leq b\}$ is called an {\em interval}. If $i <k$ and $j<l$, then the elements $a$ and $b$ are called {\em diagonal} corners and $(i,l)$ and $(k,j)$ are called {\em anti-diagonal}
 corners of $[a,b]$. An interval of the from $C=[a,a+(1,1)]$ is called a {\em cell} (with left lower corner $a$). The elements (corners) $a, a+(0,1), a+(1,0), a+(1,1)$ of $[a,a+(1,1)]$ are called the {\em vertices} of $C$. The sets $\{a,a+(1,0)\}, \{a,a+(0,1)\},  \{a+(1,0),  a+(1,1)\}$ and   $\{a+(0,1),  a+(1,1)\}$ are called the {\em edges} of $C$. We denote the set of edge of $C$ by $E(C)$.

Let $\Pc$ be a finite collection of cells of $\NN^2$. The the vertex set of $\Pc$, denoted by $V(\Pc)$ is given by $V(\Pc)=\bigcup_{C \in \Pc} V(C)$.  The edge set of $\Pc$, denoted by $E(\Pc)$ is given by $E(\Pc)=\bigcup_{C \in \Pc} E(C)$. Let $C$ and $D$ be two cells of $\Pc$. Then $C$ and $D$ are said to be {\em connected}, if there is a sequence of cells $\mathcal{C}:C= C_1, \ldots, C_m =D$ of $\Pc$  such that $C_i \cap C_{i+1}$ is an edge of $C_i$ for $i=1, \ldots, m-1$. If in addition, $C_i \neq C_j$ for all $i \neq j$, then $\mathcal{C}$ is called a {\em path} (connecting $C$ and $D$). The collection of cells $\Pc$ is called a {\em polyomino} if any two cells of  $\Pc$ are connected, see Figure~\ref{polyomino}.

 \begin{figure}[htbp]
\includegraphics[width = 3cm]{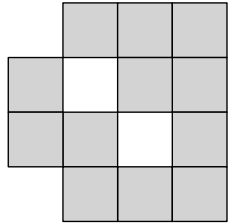}
\caption{polyomino}\label{polyomino}
\end{figure}

Let $\Pc$ be a polyomino, and let $K$ be a field. We denote by $S$ the polynomial ring over $K$ with variables $x_{ij}$ with $(i,j)\in V(\Pc)$.  Following \cite{Q} a $2$-minor $x_{ij}x_{kl}-x_{il}x_{kj}\in S$ is called an {\em inner minor} of $\Pc$ if all the cells $[(r,s),(r+1,s+1)]$ with $i\leq r\leq k-1$ and $j\leq s\leq l-1$ belong to $\Pc$. In that case the interval $[(i,j),(k,l)]$ is called an {\em inner interval} of $\Pc$. The ideal $I_\Pc\subset S$ generated by all inner minors of $\Pc$ is called the {\em polyomino ideal} of $\Pc$. We also set $K[\Pc]=S/I_\Pc$.

Let $\Pc$ be a polyomino. Following \cite{HQS},  an interval $[a,b]$ with $a=(i,j)$ and $b=(k,l)$ is called a {\em horizontal  edge interval} of $\Pc$  if $j=l$ and  the sets $\{(r,j),(r+1,j\}$ for $r=i,\ldots,k-1$ are edges of cells of $\Pc$. If a horizontal edge interval of $\Pc$ is not strictly contained in any other horizontal edge interval of $\Pc$, then we call it {\em maximal} horizontal edge interval. Similarly one defines vertical edge intervals and maximal vertical edge intervals of $\Pc$.

Let $\{V_1,\ldots,V_m\}$ be the set of maximal vertical edge intervals and $\{H_1,\ldots,H_n\}$ be the set of maximal horizontal edge intervals of $\P$. We denote by $G(\P)$, the associated bipartite graph of $\P$ with vertex set $\{v_1,\ldots,v_m\} \bigsqcup \{ h_1,\ldots,h_n\}$ and the edge set defined as follows
\[
E(G(\P)) = \{\{v_i,h_j\} \mid V_i \cap H_j \in V( \P)\}.
\]

\begin{Example}
The Figure~\ref{interval} shows a polyomino $\P$ with maximal vertical and maximal horizontal edge intervals labelled as $\{V_1, \ldots, V_4\}$ and $\{H_1, \ldots, H_4\}$ respectively, and Figure~\ref{graph}  shows the associated bipartite graph $G(\P)$ of $\Pc$.

\begin{figure}[htbp]
\includegraphics[width = 3cm]{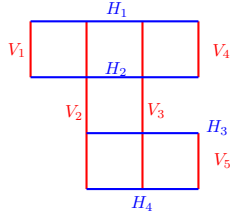}
\caption{maximal intervals of $\P$}\label{interval}
\end{figure}
\begin{figure}[htbp]
\begin{center}
\includegraphics[width=80mm]{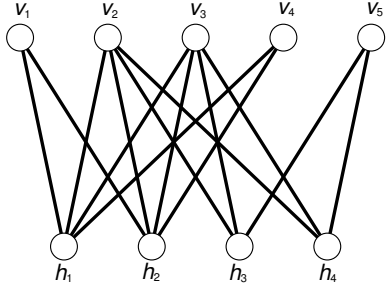}
\end{center}
\caption{associate bipartite graph of $\P$}\label{graph}
\end{figure}

\end{Example}

Let $S$ be the polynomial ring over field $K$ with variables $x_{ij}$ with $(i,j) \in V(\P)$. Note that $|V_p \cap H_q| \leq 1$. If $V_p \cap H_q = \{(i,j)\}$, then we may write $x_{ij} = x_{V_p \cap H_q}$, when required. To each cycle $\C: v_{i_1},h_{j_1}, v_{i_2}, h_{j_2}, \ldots, v_{i_r}, h_{j_r}$ in $G(\P)$, we associate a binomial in $S$ given by
$f_{\C}= x_{V_{i_1}\cap H_{j_1}} \ldots x_{V_{i_r}\cap H_{j_r}} -  x_{V_{i_2}\cap H_{j_1}} \ldots  x_{V_{i_1}\cap H_{j_r}}$. 

We recall the definition of a cycle in $\P$ from \cite{HQS}. A sequence of  vertices $\Cc_{\P}= a_1,a_2, \ldots, a_m$ in  $V(\Pc)$ with $a_m = a_1$ and such that $a_i \neq a_j$ for all  $1 \leq i < j \leq m-1$ is a called a {\em cycle}  in $\Pc$ if the following conditions hold:
\begin{enumerate}
\item[(i)]  $[a_i, a_{i+1}] $ is a horizontal or vertical edge interval of $\Pc$ for all $i= 1, \ldots, m-1$;
\item[(ii)] for $i=1, \ldots, m$ one has: if $[a_i, a_{i+1}]$ is a horizontal edge interval of $\Pc$, then  $[a_{i+1}, a_{i+2}]$  is a vertical edge interval of $\Pc$ and vice versa. Here,  $a_{m+1} = a_2$.
\end{enumerate}

We set $V(\C_{\Pc})= \{ a_1, \ldots, a_m\}$. Given a cycle $\Cc_{\P}$ in $\P$, we attach to $\Cc_{\P}$ the binomial
\[
f_{\C_{\P}} = \prod_{i=1}^{(m-1)/2} x_{a_{2i-1}} - \prod_{i=1}^{(m-1)/2} x_{a_{2i}}
\]

Moreover, we call a cycle in $\P$ is {\em primitive} if each maximal interval of $\P$ contains at most two vertices of $\Cc_{\Pc}$.

Note that if $\C: v_{i_1},h_{j_1}, v_{i_2}, h_{j_2}, \ldots, v_{i_r}, h_{j_r}$ defines a cycle in $G(\P)$ then the sequence of vertices $\C_{\P}: V_{i_1}\cap H_{j_1},  V_{i_2}\cap H_{j_1},  V_{i_2}\cap H_{j_2}, \ldots,  V_{i_r}\cap H_{j_r}, V_{i_1}\cap H_{j_r}$ is a primitive cycle in $\P$ and vice versa. Also, $f_{C}= f_{\C_{\P}} $.

We set $K[G(\P)]=K[v_ph_q \mid \{p,q\} \in E(G(\P))] \subset T= K[v_1, \ldots,v_m, h_1, \ldots, h_n]$. The subalgebra $K[G(\P)]$ is called the edge ring of $G(\P)$. Let $\varphi : S \rightarrow T$ be the surjective K-algebra homomorphism defined by $\varphi(x_{ij} ) = v_ph_q$, where $\{(i,j)\} = V_p \cap H_q$. We denote by  $J_{\P}$, the toric ideal of $K[G(\P)]$. It is known from \cite{OH2}, that $J_{\P}$ is generated by the binomials associated with cycles in $G(\P)$.

\section{Simple polyominoes are prime}

Let $\Pc$ be a polyomino and let $[a,b]$ an interval with the property that $\Pc\subset [a,b]$.  According to \cite{Q}, a polyomino $\Pc$  is called {\em simple}, if for any cell $C$ not belonging to $\Pc$ there exists a path $C=C_1,C_2,\ldots,C_m=D$ with  $C_i\not \in \Pc$ for $i=1,\ldots,m$ and such that $D$ is not a cell of  $[a, b]$. 
For example, the polyomino illustrated in Figure~\ref{polyomino} is not simple but the one in Figure~\ref{simple} is simple.
It is conjectured in \cite{Q} that $I_\Pc$ is a prime ideal if $\Pc$ is simple.
 
\begin{figure}[htbp]
\includegraphics[width = 3cm]{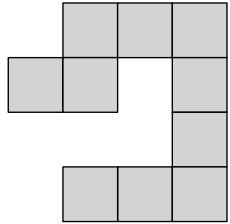}
\caption{simple polyomino}\label{simple}
\end{figure}

We recall from graph theory that a graph is called weakly chordal if every cycle of length greater than 4 has a chord. In order to prove following lemma, we define some notation. We call a cycle $\C_{\P}: V_{i_1}\cap H_{j_1},  V_{i_2}\cap H_{j_1},  V_{i_2}\cap H_{j_2}, \ldots,  V_{i_r}\cap H_{j_r}, V_{i_1}\cap H_{j_r}$ in $\P$ has a self crossing if for some $i_p \in \{i_1, \ldots, i_{r-1}\}$ and $j_q \in \{j_1, \ldots, j_{s-1}\}$ there exist vertices $a=V_{i_p} \cap H_{i_p},  b=V_{i_p} \cap H_{i_{p+1}}, c=V_{i_s} \cap H_{i_s}, d=V_{i_{s+1}} \cap H_{i_s}$ such that there exists a vertex $e \notin \{a,b,c,d\}$ such that $e \in [a,b] \cap [c,d]$. In this situation $e = V_{i_p} \cap H_{i_s}$. If $\C$ is the associated cycle in $G(\P)$ then It also shows that $\{v_{i_p}, h_{i_s}\} \in E(G(\P))$ and it gives us a chord in $\C$. 

Let  $\C_{\P}:a_1, a_2, \ldots, a_r$ be a cycle in $\P$ which does not have any self crossing. Then we call  the  area bounded by edge intervals $[a_i, a_{i+1}]$ and $[a_r, a_1]$ for $i \in \{1, r-1\}$, the {\em interior } of $\C_{\P}$. Moreover, we call a cell $C$ is an {\em interior} cell of $\C_\P$ if $C$ belongs to the interior of $\C_{\P}$.

\begin{Lemma}
Let $\P$ be a simple polyomino. Then the graph $G(\P)$ is weakly chordal.
\end{Lemma}

\begin{proof}
Let $\C$ be a cycle of $G(\P)$ of length $2n$ with $n \geq 3$ and $\C_{\P}$ be the associated primitive cycle in $\P$. We may assume that $\C_{\P}$ does not have any self crossing. Otherwise, by following the definition of self crossing, we know that $\C$ has a chord.

 Let  $\C: v_{i_1},h_{j_1}, v_{i_2}, h_{j_2}, \ldots, v_{i_r}, h_{j_r}$ and $\C_{\P}: V_{i_1}\cap H_{j_1},  V_{i_2}\cap H_{j_1},  V_{i_2}\cap H_{j_2}, \ldots,  V_{i_r}\cap H_{j_r}, V_{i_1}\cap H_{j_r}$.  We may write $a_1= V_{i_1}\cap H_{j_1},  a_2 =V_{i_2}\cap H_{j_1},  a_3=V_{i_2}\cap H_{j_2}, \ldots,  a_{2r-1}=V_{i_r}\cap H_{j_r}, a_{2r}= V_{i_1}\cap H_{j_r}$. Also, we may assume that $a_1$ and $a_2$ belongs to the same maximal horizontal edge interval. Then $a_{2r}$ and $a_1$ belongs to  the same maximal vertical edge interval.

First, we show that every interior cell of $\C_\P$  belongs to $\P$. Suppose that we have an interior cell $C$ of $\C_{\P}$ which does not belong to $\P$. Let $\mathcal{J}$ be any interval such that $\P \subset \mathcal{J}$. Then, by using the definition of simple polyomino, we obtain a path of cells $C=C_1, C_2, \ldots, C_t$ with $C_i \notin P$, $i=1, \ldots t$ and $C_t$ is a boundary cell in $\mathcal{J}$. It shows that $V(C_1) \cup V(C_2) \cup \ldots \cup V( C_t)$ intersects atleast one of $[a_i, a_{i+1}]$ for $i\in \{1, \ldots, r-1\}$ or $[a_r, a_1]$, which is not possible because $\C_{\P}$ is a cycle in $\P$. Hence $C \in \P$. It shows that an interval in interior of $\C_{\P}$ is an inner interval of $\P$.

Let $\mathcal{I}$ be the maximal inner interval of $\C_\P$ to which $a_1$ and $a_2$ belongs and let $b,c$ the 
corner vertices of $\mathcal{I}$. We may assume that  $a_1$ and $c$ are the diagonal corner and $a_2$ and $b$ are the anti-diagonal corner of $\mathcal{I}$. If $b,c \in V(\C_{\P})$ then primitivity of  $\C$  implies that $\C$ is a cycle of length 4. We may assume that $b \notin V(\C_{\P})$. Let $H'$ be the maximal horizontal edge interval which contains $b$ and $c$. The maximality of $\mathcal{I}$ implies that $H' \cap V(\C_{\P}) \neq \emptyset$. For example, see  Figure~\ref{something2}. Therefore, $\{v_{i_1}, h'\}$ is a chord in $\C$, as desired.

\begin{figure}[htbp]
\includegraphics[width = 3cm]{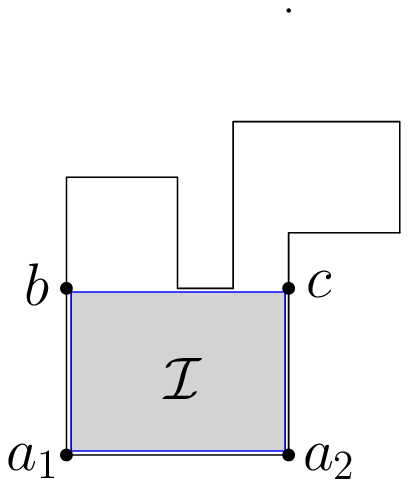}
\caption{maximal inner interval}\label{something2}
\end{figure}

\end{proof}

\begin{Theorem}
Let $\P$ be a simple polyomino. Then $I_\P = J_\P$.
\end{Theorem}
\begin{proof}
First we show that $I_{\P} \subset J_{\P}$. Let $f=x_{ij} x_{kl} - x_{il}x_{kj} \in I_{\P}$. Then there exist maximal vertical edge intervals $V_p$ and $ V_q$ and maximal horizontal edge intervals $H_m$ and $H_n$ of $\P$ such that $(i,j), (i,l) \in V_p$, $(k,j), (k,l) \in V_q$ and $(i,j), (k,j) \in H_m$, $(i,l), (k,l) \in H_n$. It gives that $\phi(x_{ij} x_{kl}) = v_ph_mh_nv_q= \phi (x_{il} x_{kj})$. This shows $f \in J_P$. 

Next, we show that $J_{\P} \subset I_{\P}$. It is known from \cite{OH1} and \cite{OH2} that toric ideal of weakly chordal bipartite graph is minimally generated by quadratic binomials associated with cycles of length 4.  It suffices to show that  $f_{\C} \in I_P$ where $\C$ is a cycle of length 4 in $G(\P)$. 

Let $\Ic$ be an interval such that $\P \subset \I$. Let $\C:h_1, v_1, h_2, v_2$.. Then, $\C_{\P}: a_{11}=H_1 \cap V_1, a_{21}=H_2 \cap V_1$, $a_{22}=H_2 \cap V_2$ and $a_{12}=H_1 \cap V_2$ is the associated cycle in $\P$ which also determine an  interval in $\I$. Let $a_{11}$ and $a_{22}$ be the diagonal corners of this interval. We need to show that $[a_{11}, a_{22}]$ is an inner interval in $\P$. Assume that $[a_{11}, a_{22}]$ is not an inner interval of $\P$, that is, there exist a cell $C \in [a_{11}, a_{22}]$ which does not belong to $\P$. Using the fact that $\P$ is a simple polyomino, we obtain a path of cells $C= C_1, C_2, \ldots, C_r$ with $C_i \notin \P$, $i=1, \ldots, r$ and $C_r$ is a cell in $\I$.
Then, $V(C_1 \cup \ldots \cup C_r) $ intersects atleast one of the maximal intervals $H_1, H_2, V_1, V_2$, say $H_1$, which contradicts the fact that $H_1$ is an interval in $\P$. Hence, $[a_{11}, a_{22}]$ is an inner interval of $\P$ and $f_{\C} \in I_{\P}$.
\end{proof}

\begin{Corollary}
Let $\P$ be a simple polyomino. Then $K[\Pc]$ is Koszul and a normal Cohen--Macaulay domain.
\end{Corollary}

\begin{proof}
From \cite{OH1}, we know that $J_{\Pc}=I_{\Pc}$ has squarefree quadratic Gr\"obner basis with respect to a suitable monomial order.
Hence $K[\Pc]$ is Koszul. By theorem of Sturmfels \cite{St}, one obtains that $K[\Pc]$ is normal and then following a theorem of Hochster \cite[Theorem 6.3.5]{BH}, we obtain that $K[\Pc]$ is Cohen–-Macaulay.
\end{proof}
\medskip
A polyomino ideal may be prime even if the polyomino is not simple.  The polyomino ideal attached to the polyomino in Figure~\ref{prime} is prime. However, the polyomino ideal attached to the polyomino attached in Figure~\ref{notprime} is not prime. It would be interesting to know a complete characterization of polyominoes whose attached polyomino ideals are prime, but it is not easy to answer. However, as a first step, it is already an  interesting question to classify polyominoes with only ``one hole'' such that their associated polyomino ideal is prime. 
\begin{figure}[htbp]
\includegraphics[width = 3cm]{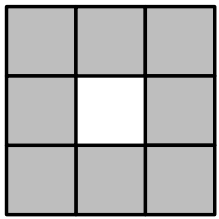}
\caption{polyomino with prime polyomino ideal}\label{prime}
\end{figure}
\begin{figure}[htbp]
\begin{center}
\includegraphics[width=4cm]{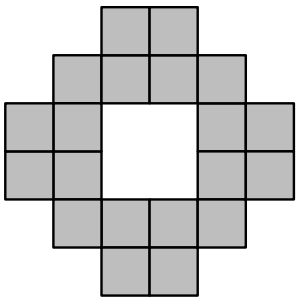}
\end{center}
\caption{polyomino with ``one hole''}\label{notprime}
\end{figure}

\end{document}